\documentclass[reqno]{amsart}
\usepackage{amsmath,amssymb,amsthm,amscd}
\newtheorem{theorem}{Theorem}[section]
\newtheorem{lemma}[theorem]{Lemma}
\newtheorem{proposition}[theorem]{Proposition}

\theoremstyle{definition}
\newtheorem{remark}[theorem]{Remark}

\newcommand{\ZZ}{\mathbb{Z}}

\newcommand{\RR}{\mathbb{R}}
\newcommand{\CC}{\mathbb{C}}

\newcommand{\matX}{\textsl{Mat}_X(\CC^*)}         
\newcommand{\matY}{\textsl{Mat}_Y(\CC^*)}         

\DeclareMathOperator{\LCM}{LCM}

\newcommand{\nexteq}{\displaybreak[0]\\ &=}

\begin{document}

\title{Spin models constructed from Hadamard matrices}

\author[T. Ikuta]{Takuya Ikuta}
\address{Faculty of Law, Kobe Gakuin University,
Minatojima, Chuo-ku, Kobe, 650-8586 Japan}
\email{ikuta@law.kobegakuin.ac.jp}

\author[A. Munemasa]{Akihiro Munemasa}
\address{Graduate School of Information Sciences,
Tohoku University, Sendai, 980-8579 Japan}
\email{munemasa@math.is.tohoku.ac.jp}

\subjclass[2000]{05B20,05E30}

\keywords{spin model, association scheme, Hadamard matrix}

\date{January 6, 2011}

\begin{abstract}
A spin model (for link invariants) is a square matrix $W$
which satisfies certain axioms.
For a spin model $W$, it is known that 
$W^TW^{-1}$ is 
a permutation matrix, and its order is called the index of $W$.
F.~Jaeger and K.~Nomura found spin models of index $2$, by 
modifying the construction of symmetric spin models from
Hadamard matrices.

The aim of this paper is to give a construction of 
spin models of an arbitrary even index
from any Hadamard matrix. 
In particular, we show that our spin models of indices a power of $2$ are new.
\end{abstract}

\maketitle


\section{Introduction}
The notion of spin model was introduced by V.F.R. Jones
\cite{Jo:pac} to construct invariants of  knots and links.
The original definition due to Jones requires that a spin model
be a symmetric matrix, but later 
by K.~Kawagoe, A.~Munemasa, and Y.~Watatani \cite{KMW},
a general definition allowing non-symmetric matrices is given.
In this paper, we consider spin models which are not
necessarily symmetric.

Let $X$ be a non-empty finite set.
We denote by $\matX$ the set of square matrices with non-zero complex entries
whose rows and columns are indexed by $X$. 
For $W \in \matX$ and 
$x$, $y \in X$, the $(x,y)$-entry of $W$ is denoted by $W(x,y)$.
A spin model $W \in \matX$ is defined to be a matrix 
which satisfies two conditions (type~II and
type~III; see Section~\ref{sec:2}).

One of the examples of spin models is a Potts model, defined as follows.
Let $X$ be a finite set with $r$ elements, and let $I, J \in \matX$ be the identity matrix and the all $1$'s matrix,
respectively. Let $u$ be a complex number satisfying
\begin{equation}        \label{potts}
\begin{split}
(u^2+u^{-2})^2=r \text{ if $r\geq 2$,}\\
u^4=1 \text{ if $r=1$.}
\end{split}
\end{equation}
Then a Potts model $A_u$ is defined as
\[
A_u=u^3I-u^{-1}(J-I).
\]

As examples of spin models, we know only Potts models \cite{Jo:pac,J}, 
spin models on finite abelian groups \cite{BBJ,BM},
Jaeger's Higman-Sims model \cite{J},
Hadamard models \cite{N:94,JN}, 
non-symmetric Hadamard models \cite{JN}, 
and tensor products of these.
Apart from spin models on finite abelian groups, 
non-symmetric Hadamard models are  
essentially the only known family of non-symmetric spin models.

If $W$ is a spin model, then 
by \cite[Proposition 2]{JN},
$R=W^TW^{-1}$ is a permutation matrix.
The order of $R$ as a permutation 
is called the {\em index\/} of the spin model $W$.

A {\em Hadamard matrix\/} of order $r$
is a square matrix $H$ of size $r$ with entries 
$\pm 1$ satisfying $HH^T=I$. 
In \cite{JN}, F.~Jaeger and K.~Nomura 
constructed {\em non-symmetric Hadamard models,\/}
which are spin models of index $2$:
\begin{equation}        \label{shsm}
W =
   \begin{array}{r@{}l}
      \left(
        \begin{array}{cc}
        \left(
        \begin{array}{rr}
        1 & 1 \\
        1 & 1
        \end{array}
        \right) \otimes A_u &
        \left(
        \begin{array}{rr}
        1 & -1 \\
        -1 & 1
        \end{array}
        \right) \otimes \xi H \\
        \left(
        \begin{array}{rr}
         -1 & 1 \\
        1 & -1
        \end{array}
        \right) \otimes \xi H^T &
        \left(
        \begin{array}{rr}
        1 & 1 \\
        1 & 1
        \end{array}
        \right) \otimes A_u
        \end{array}
      \right),
  \end{array}
\end{equation}
where $\xi$ is a primitive $8$-th root of unity,
$A_u \in \matX$ is a Potts model,
and $H \in \matX$ is a Hadamard matrix.

Note that non-symmetric Hadamard models are a modification
of the earlier Hadamard models (\cite{JN}, see also \cite[Section 5]{JN}),
defined by
\begin{equation}        \label{nshsm}
W' =
   \begin{array}{r@{}l}
      \left(
        \begin{array}{cc}
        \left(
        \begin{array}{rr}
        1 & 1 \\
        1 & 1
        \end{array}
        \right) \otimes A_u &
        \left(
        \begin{array}{rr}
        1 & -1 \\
        -1 & 1
        \end{array}
        \right) \otimes \omega H \\
        \left(
        \begin{array}{rr}
         1 & -1 \\
        -1 & 1
        \end{array}
        \right) \otimes \omega H^T &
        \left(
        \begin{array}{rr}
        1 & 1 \\
        1 & 1
        \end{array}
        \right) \otimes A_u
        \end{array}
      \right),
  \end{array}
\end{equation}
where $\omega$ is a $4$-th root of unity.

To construct spin models of index $m>2$, 
it seems natural to consider an $m\times m$ block
matrix $W=(W_{i,j})_{i,j\in \ZZ_m}$ such that each block $W_{ij}$ is the 
tensor product of
two matrices like those in (\ref{shsm}) and (\ref{nshsm}):
\begin{equation}  \label{eq:5}
   W_{ij} = S_{ij} \otimes T_{ij}
   \qquad (i,j \in \ZZ_m).
\end{equation}
Such matrices appeared in \cite[Proposition 6.2]{IN},
with the matrices $S_{ij}\in\textsl{Mat}_{\ZZ_m}(\CC^*)$ given by
\begin{equation}  \label{eq:6}
   S_{ij}(\ell, \ell') = \eta^{(\ell-\ell')(i-j)}
   \qquad (\ell,\, \ell' \in \ZZ_m),
\end{equation}
where $\eta$ is a primitive $m$-th root of unity.

In this paper, we construct an infinite class of spin models of even index
containing non-symmetric Hadamard models.
Also, we construct an infinite class of symmetric spin models
containing Hadamard models.
Our main result is as follows:


\begin{theorem}  \label{thm:main1}
Let $r$ be a positive integer, and
let $m$ be an even positive integer.
Define $Y=\{1,\dots,r\}$, 
$X_i=\{ (i,\ell,x) \mid \ell \in \ZZ_m,\; x \in Y \}$ for 
$i \in \ZZ_m$,
and $X=X_0 \cup \cdots \cup X_{m-1}$.
Let $A_u$, $H \in \matY$ be a Potts model and a Hadamard matrix,
respectively.
Define $V_{ij}$ for $i,j\in\ZZ_m$ by
\begin{equation}  \label{con1}
V_{ij}=
\left\{
  \begin{array}{ll}
  A_u & \mbox{ if \ $i-j$ is even, } \\
  H & \mbox{ if \ $(i,j) \equiv (0,1) \pmod 2$, } \\
  H^T & \mbox{ if \ $(i,j) \equiv (1,0) \pmod 2$. }
  \end{array}
\right.
\end{equation}
Then the following statements hold:
\begin{itemize}
\item[{\rm (i)}] 
Let $a$ be a primitive $2m^2$-th root of unity.
Let $W \in \matX$ be the matrix
whose $(\alpha,\beta)$ entry is given by 
$a^{2m(\ell-\ell')(i-j)+\epsilon(i,j)} V_{ij}(x,y)$
for $\alpha=(i,\ell,x), \beta=(j,\ell',y) \in X$,
where $\epsilon(i,j)=(i-j)^2+m(i-j)$.
Then $W$ is a spin model of index $m$.
\item[{\rm (ii)}] 
Let $\eta$ be a primitive $m$-th root of unity, and let
$b$ be an $m^2$-th root of unity.
Let $W' \in \matX$ be the matrix
whose $(\alpha,\beta)$ entry is given by 
$\eta^{(\ell-\ell')(i-j)} b^{\delta(i,j)} V_{ij}$
for $\alpha=(i,\ell,x), \beta=(j,\ell',y) \in X$,
where $\delta(i,j)=(i-j)^2$.
Then $W'$ is a symmetric spin model.
\end{itemize}
\end{theorem}

Note that, in order for $a^{\epsilon(i,j)}$ 
and $b^{\delta(i,j)}$ to be well-defined,
we need to identify $\ZZ_m$ with the subset $\{0,1,\dots,m-1\}$ of integers.

\begin{remark}  \label{0902skype}
In Theorem~\ref{thm:main1} (i), if we define $S_{ij}$ by
(\ref{eq:6}) with $\eta=a^{2m}$, and $T_{ij}$ by
$T_{ij}=a^{\epsilon(i,j)}V_{ij}$, then the
$(X_i,X_j)$-block of the matrix $W$ is given by (\ref{eq:5}).
Similarly, in Theorem~\ref{thm:main1} (ii),
(\ref{eq:5}) holds with 
$T_{ij}=b^{\delta(i,j)}V_{ij}$.
\end{remark}

The spin models
$W$, $W'$ given in Theorem~\ref{thm:main1} are determined by
a Hadamard matrix $H$ of order $r$,
a complex number $u$ satisfying (\ref{potts}),
and a primitive $2m^2$-th root of unity $a$
or an $m^2$-th root of unity $b$, respectively.
Throughout this paper,
we denote by $W_{H,u,a}$, $W'_{H,u,b}$ the spin models given 
by Theorem~\ref{thm:main1} (i), (ii), respectively.

Observe that, for any spin models $W_i$ ($i=1,2$) of indices $m_i$,
their tensor product $W_1 \otimes W_2$ is also a spin model of 
index $\LCM(m_1,m_2)$.
In Section~\ref{sec:5},
we show that the non-symmetric spin model 
$W_{H,u,a}$ 
whose index is a power of $2$ is new in the following sense:


\begin{theorem}  \label{thm:main2}
Let $H$ be a Hadamard matrix of order $r$.
Let $W_{H,u,a}$ be a spin model given in 
Theorem~{\rm\ref{thm:main1} (i)},
whose index $m$ is a power of\/ $2$.
If $r>4$, then $W_{H,u,a}$ cannot be decomposed into a 
tensor product of known spin models.
\end{theorem}
We note that the list of known spin models is given in Section $5$.
Jaeger and Nomura \cite[p.278]{JN} expected that new 
non-symmetric spin models of index a power of $2$ should be found,
and our results confirm this expectation.


\section{Type~II and Type~III conditions on block matrices of tensor products}\label{sec:2}

First we define a spin model.
A {\em type~II matrix\/} on a finite set $X$ is a matrix $W \in \matX$ 
which satisfies the {\em type~II condition}:

\begin{equation}  \label{type2}
   \sum_{x \in X} \frac{W(\alpha,x)}{W(\beta,x)} = n \delta_{\alpha,\beta}
   \qquad (\mbox{for all $\alpha, \beta \in X$}).
\end{equation}
Let $W^- \in \matX$ be defined by $W^-(x,y)=W(y,x)^{-1}$.  Then
the type~II condition is written as $W W^- = n I$.
Hence, if $W$ is a type~II matrix, then $W$ is non-singular with
$W^{-1} = n^{-1}W^-$.

A type~II matrix $W \in \matX$ is called a {\em spin model\/} if 
$W$ satisfies the {\em type~III condition\/}:
\begin{equation}  \label{type3}
  \sum_{x \in X} \frac{W(\alpha,x)W(\beta,x)}{W(\gamma,x)}
  = D \frac{W(\alpha,\beta)}{W(\alpha,\gamma)W(\gamma,\beta)}
  \qquad (\mbox{for all $\alpha,\beta,\gamma \in X$})
\end{equation}
for some nonzero real number $D$ with $D^2=n$,
which is independent of the choice of $\alpha,\beta,\gamma\in X$.

Let $m$ be a positive integer.
In this section, assuming that $W$ is an $m\times m$
block matrix with blocks of the form (\ref{eq:5}),
we will establish conditions on $T_{ij}$
under which $W$ satisfies the type~II and type~III conditions.
Some parts of these conditions are already given in
\cite[Proposition 5.1, Proposition 6.2]{IN}.

Let $\eta$ be a primitive $m$-th root of unity,
and let $S_{ij}$ be the matrix of size $m$ defined by
(\ref{eq:6}) for $i,j \in \ZZ_m$.
Let $r$ be a positive integer,
and define $Y=\{1,\dots,r\}$, 
$X_i=\{ (i,\ell,x) \mid \ell \in \ZZ_m,\; x \in Y \}$ for 
$i \in \ZZ_m$,
and $X=X_0 \cup \dots \cup X_{m-1}$.
Let $T_{ij} \in \matY$ be a matrix for $i,j \in \ZZ_m$,
and let $W_{ij}$ be the matrix defined by (\ref{eq:5}).
Let $W\in\matX$ be the matrix 
whose $(X_i, X_j)$-block is $W_{ij}$ for $i,j \in \ZZ_m$.
Then 
\begin{equation}  \label{eq:25}
W((i,\ell,x),(j,\ell',y))=S_{ij}(\ell,\ell')T_{ij}(x,y).
\end{equation}


\begin{lemma}[{\cite[Proposition 5.1]{IN}}]  \label{lemma:2.1}
The matrix $W$ is a type~II matrix if and only if $T_{ij}$ 
is a type~II matrix for all $i,j \in\ZZ_m$.
\end{lemma}


\begin{lemma}  \label{lemma:2.2}
The matrix $W$ satisfies the type~III condition {\rm(\ref{type3})} if and only if
the following equality holds for all $i_1,i_2,i_3 \in\ZZ_m$ 
and $x_1, x_2, x_3 \in Y$:
\begin{equation}  \label{eq:21}
\sum_{x \in Y}
\frac{ T_{i_1,i_0}(x_1, x)T_{i_2,i_0}(x_2, x) }
{ T_{i_3,i_0}(x_3, x) }
=\frac{D}{m}\cdot \frac{ T_{i_1,i_2}(x_1,x_2) }
{ T_{i_1,i_3}(x_1,x_3) T_{i_3,i_2}(x_3,x_2) },
\end{equation}
where $i_0=i_1+i_2-i_3 \bmod{m}$.
\end{lemma}

\begin{proof}
The type~III condition (\ref{type3}) for $\alpha=(i_1,\ell_1,x_1)$,
$\beta=(i_2,\ell_2,x_2)$, $\gamma=(i_3,\ell_3,x_3)$
is equivalent to
\begin{align*}
& \sum_{i,\ell \in\ZZ_m} 
\frac{ \eta^{(\ell_1-\ell)(i_1-i)} \eta^{(\ell_2-\ell)(i_2-i)} }{ \eta^{(\ell_3-\ell)(i_3-i)} }
\sum_{x \in Y} \frac{ T_{i_1,i}(x_1, x)T_{i_2,i}(x_2, x) }
{ T_{i_3,i}(x_3, x) } \\
& = D \frac{ \eta^{(\ell_1-\ell_2)(i_1-i_2)} }
{ \eta^{(\ell_1-\ell_3)(i_1-i_3)}\eta^{(\ell_3-\ell_2)(i_3-i_2)} }
\cdot\frac{ T_{i_1,i_2}(x_1,x_2) }
{ T_{i_1,i_3}(x_1,x_3) T_{i_3,i_2}(x_3,x_2) }.
\end{align*}
By a direct computation, we obtain
\begin{align*}
 & \frac{ \eta^{(\ell_1-\ell)(i_1-i)} \eta^{(\ell_2-\ell)(i_2-i)} }
{ \eta^{-(\ell_3-\ell)(i_3-i)} } \cdot
\frac{ \eta^{(\ell_1-\ell_3)(i_1-i_3)}\eta^{(\ell_3-\ell_2)(i_3-i_2)} }
{ \eta^{(\ell_1-\ell_2)(i_1-i_2)} } \\
 & = \eta^{(\ell_1+\ell_2-\ell_3-\ell)(i_1+i_2-i_3-i)}.
\end{align*}
So (\ref{type3}) is equivalent to
\begin{align}
&  \sum_{i \in\ZZ_m} (\sum_{\ell \in\ZZ_m} 
\eta^{(\ell_1+\ell_2-\ell_3-\ell)(i_1+i_2-i_3-i)} ) 
\sum_{x \in Y}
\frac{ T_{i_1,i}(x_1,x)T_{i_2,i}(x_2, x) }
    { T_{i_3,i}(x_3, x) }  \notag\\
& = D \frac{ T_{i_1,i_2}(x_1, x_2) }
      { T_{i_1,i_3}(x_1, x_3) T_{i_3,i_2}(x_3, x_2) }.
\label{eq:22}
\end{align}
Since $\eta$ is a primitive $m$-th root of unity and 
$i_0=i_1+i_2-i_3 \bmod{m}$, we have
\[
\sum_{\ell \in\ZZ_m} 
\eta^{(\ell_1+\ell_2-\ell_3-\ell)(i_1+i_2-i_3-i)}
=m\delta_{i,i_0}.
\]
Thus (\ref{eq:22}) 
is equivalent to (\ref{eq:21}).
\end{proof}

We remark that in \cite[Proposition 6.2]{IN} only the 
necessity of (\ref{eq:21}) for the type~III condition is proved.

Let $z_m$ be the permutation matrix of order $m$:
\[
z_m=\left(
    \begin{array}{rrrr}
      & & & 1 \\
    1 & & & \\
      & \ddots & & \\
      & & 1 &
    \end{array}
    \right).
\]
We define the permutation matrix $R$ of size $n=m^2r$
by $R=I_m \otimes z_m \otimes I_r$,
where $I_m$ and $I_r$ are the identity matrices of size $m$ and $r$, respectively.
The order of $R$ is $m$.


\begin{lemma}  \label{lemma:2.3}
The matrix $W$ satisfies $W^TW^{-1}=R$ if and only if
$T_{ij}=\eta^{i-j} T_{ji}^T$ holds for all $i,j \in\ZZ_m$.
\end{lemma}
\begin{proof}

For
$\alpha=(i,\ell,x)$ and $\beta=(j,\ell',y) \in X$,
\begin{align*}
W^T(\alpha, \beta) &= W(\beta,\alpha) 
\nexteq
\eta^{(\ell'-\ell)(j-i)} T_{j,i}(y, x), 
\displaybreak[0]\\
(RW)(\alpha, \beta) &= 
( (I_m \otimes z_m \otimes I_r)W )
((i,\ell,x), (j,\ell',y))
\nexteq
W((i,\ell-1,x), (j,\ell',y)) 
\nexteq 
\eta^{(\ell-1-\ell')(i-j)} T_{ij}(x, y)
\nexteq 
\eta^{(\ell'-\ell)(j-i)}\eta^{-(i-j)} 
T_{ij}(x, y).
\end{align*}
Therefore $R=W^T W^{-1}$ if and only if
$T_{ji}(y,x)=\eta^{-(i-j)}T_{ij}(x,y)$ holds for
all $i,j\in\ZZ_m$ and $x,y\in Y$.
\end{proof}



\section{Proof of Theorem~\ref{thm:main1}}

From Remark~\ref{0902skype},
the results in Section~2 can be used for the matrices
$W$ and $W'$ given in Theorem~\ref{thm:main1},
if we define $T_{ij}$ according to Remark~\ref{0902skype}.


For a mapping $g$ from $\ZZ^2$ to $\ZZ$,
we denote by $\lambda_g$ the mapping from 
$\ZZ^4$ to $\ZZ$ defined by
\begin{equation}  \label{eq:31}
\lambda_g(i_1,i_2,i_3,i_4)
=g(i_1,i_4)+g(i_2,i_4)-g(i_3,i_4)+g(i_1,i_3)+g(i_3,i_2)-g(i_1,i_2).
\end{equation}
Recall that we regard $\ZZ_m$ as the subset
$\{0,1,\dots,m-1\}$ of $\ZZ$, and
$\delta,\epsilon:\ZZ^2\to\ZZ$ are defined by
$\delta(i,j)=(i-j)^2$, $\epsilon(i,j)=\delta(i,j)+m(i-j)$,
respectively.


\begin{lemma}  \label{lem31}
For all $i_1,i_2,i_3,i_4 \in\ZZ$, we have
\begin{eqnarray*}
\lambda_{\delta}(i_1,i_2,i_3,i_4) &=& 
( i_1+i_2-i_3-i_4 )^2, \\
\lambda_{\epsilon}(i_1,i_2,i_3,i_4) &=& 
( i_1+i_2-i_3-i_4 ) ( i_1+i_2-i_3-i_4+m ).
\end{eqnarray*}
In particular,
if $i_0=i_1+i_2-i_3 \pmod{m}$,
then 
\begin{eqnarray*}
\lambda_{\delta}(i_1,i_2,i_3,i_0) & \equiv & 0 \pmod{m^2}, \\
\lambda_{\epsilon}(i_1,i_2,i_3,i_0) & \equiv & 0 \pmod{2m^2}.
\end{eqnarray*}
\end{lemma}
\begin{proof}
Straightforward.
\end{proof}

In \cite[$\S5.1$]{JN}, the following is used to construct
non-symmetric or symmetric Hadamard models:


\begin{lemma}[{\cite[\S5.1]{JN}}]  \label{lemma:4.1}
Let $A_u$, $H \in \matY$ be a Potts model and a Hadamard matrix,
respectively.
Then the following holds for all $x_1, x_2, x_3 \in Y$:
\begin{eqnarray}
\sum_{y \in Y} \frac{A_u(x_1, y)A_u(x_2, y)}{A_u(x_3, y)} 
&=& D_u \frac{A_u(x_1, x_2)}{A_u(x_1, x_3)A_u(x_3, x_2)}, \label{eq:88} \\
\sum_{y \in Y} A_u(x_1, y)H(y, x_2)H(y, x_3) 
&=& D_u \frac{H(x_1, x_2)H(x_1, x_3)}{A_u(x_2, x_3)}, \label{eq:8} \\
\sum_{y \in Y} A_u(x_1, y)H(x_2, y)H(x_3, y) 
&=& D_u \frac{H( x_2, x_1)H(x_3, x_1)}{A_u(x_2, x_3)}, \label{eq:9} \\
\sum_{y \in Y} \frac{H(y, x_1)H(y, x_2)}{A_u(x_3, y)} 
&=& D_u A_u(x_1,x_2)H(x_3, x_1)H(x_3, x_2), \label{eq:10} \\
\sum_{y \in Y} \frac{H(x_1, y)H(x_2, y)}{A_u(x_3, y)} 
&=& D_u A_u(x_1, x_2)H(x_1, x_3)H(x_2, x_3), \label{eq:11}
\end{eqnarray}
where 
\[
D_u=
\begin{cases}
-u^2-u^{-2} &\text{if } |Y| \geq 2, \\
u^2 & \text{if } |Y|=1.
\end{cases}
\]
\end{lemma}

We now prove Theorem~\ref{thm:main1}.
Since $A_u$ and $H$ are type~II matrices,
so are the matrices $T_{ij}=a^{\epsilon(i,j)}V_{ij}$
or $b^{\delta(i,j)}V_{ij}$. Thus, Lemma~\ref{lemma:2.1}
implies that $W_{H,u,a}$ and $W'_{H,u,b}$ are
type~II matrices. 

We claim
\begin{equation}  \label{eq:41}
\sum_{y \in Y} \frac{ V_{i_1,i_0}(x_1,y)V_{i_2,i_0}(x_2,y) }
{ V_{i_3,i_0}(x_3,y) }=
D_u \frac{ V_{i_1,i_2}(x_1,x_2) }
{ V_{i_1,i_3}(x_1,x_3) V_{i_3,i_2}(x_3,x_2) }
\end{equation}
for all $i_1,i_2,i_3 \in\ZZ_m$ and $x_1,x_2,x_3 \in Y$,
where $i_0=i_1+i_2-i_3 \bmod{m}$.
Indeed, let $i_1,i_2,i_3\in\ZZ_m$. Then
\[
(\ref{eq:41}) \Longleftrightarrow
\begin{cases}
(\ref{eq:88}) & \text{if} \quad (i_1,i_2,i_3) \equiv (0,0,0), (1,1,1) \pmod 2, \\
(\ref{eq:8}) & \text{if} \quad (i_1,i_2,i_3) \equiv (0,1,1)
\pmod 2, \\
(\ref{eq:9}) & \text{if} \quad (i_1,i_2,i_3) \equiv (1,0,0)
\pmod 2, \\
(\ref{eq:10}) & \text{if} \quad (i_1,i_2,i_3) \equiv (1,1,0) \pmod 2, \\
(\ref{eq:11}) & \text{if} \quad (i_1,i_2,i_3) \equiv (0,0,1) \pmod 2.
\end{cases}
\]
Moreover, when $(i_1,i_2,i_3) \equiv (1,0,1),(0,1,0) \pmod2$,
(\ref{eq:41}) is equivalent to (\ref{eq:8}), (\ref{eq:9}),
respectively, with $x_1$ and $x_2$ switched.
Therefore, (\ref{eq:41}) holds in all cases
by Lemma~\ref{lemma:4.1}.

First, we show that $W_{H,u,a}$ and $W'_{H,u,b}$ satisfy the condition (\ref{eq:21}).
From Lemma~\ref{lem31} we have
\[
a^{\lambda_{\epsilon}(i_1,i_2,i_3,i_0)}=1 \quad \text{and} \quad b^{\lambda_{\delta}(i_1,i_2,i_3,i_0)}=1.
\]
In view of (\ref{eq:31}), these imply
\begin{equation}  \label{eq:43}
c^{ g(i_1,i_0)+g(i_2,i_0)-g(i_3,i_0) }=c^{ g(i_1,i_2)-g(i_1,i_3)-g(i_3,i_2) },
\end{equation}
where $(c,g)=(a,\epsilon),(b,\delta)$.
Combining (\ref{eq:41}) and (\ref{eq:43}),
we obtain
\begin{align*}
 & \sum_{y \in Y} \frac{ c^{ g(i_1,i_0) }V_{i_1,i_0}(x_1,y) c^{ g(i_2,i_0) }V_{i_2,i_0}(x_2,y) }
{ c^{ g(i_3,i_0) }V_{i_3,i_0}(x_3,y) } \\
 & =D_u \frac{ c^{ g(i_1,i_2) }V_{i_1,i_2}(x_1,x_2) }
{ c^{ g(i_1,i_3) }V_{i_1,i_3}(x_1,x_3) c^{ g(i_3,i_2) }V_{i_3,i_2}(x_3,x_2) }.
\end{align*}
for all $i_1,i_2,i_3 \in\ZZ_m$ and $x_1,x_2,x_3 \in Y$.
Thus (\ref{eq:21}) holds by setting $D=mD_u$.
It follows from Lemma~\ref{lemma:2.2} that
$W_{H,u,a}$ and $W'_{H,u,b}$ satisfy the type~III condition
(\ref{type3}), and hence they are spin models.
Since $\delta(i,j)=\delta(j,i)$, 
$W'_{H,u,b}$ is symmetric.


Finally, we show that $W_{H,u,a}$ has index $m$.
Since $a^{2m}=\eta$,
we have $a^{\epsilon(i,j)-\epsilon(j,i)}=a^{2m(i-j)}=\eta^{i-j}$.
So, $T_{ij}=\eta^{i-j}T_{ji}^T$ holds for all $i,j\in\ZZ_m$.
From Lemma~\ref{lemma:2.3}, $W_{H,u,a}$ has index $m$.
This completes the proof of Theorem~\ref{thm:main1}.


\section{Properties of spin models in Theorem~\ref{thm:main1}}

For a positive integer $r$, we let $u$ be a complex number satisfying (\ref{potts}).



\begin{lemma} \label{lemma:5.2}
If $r\leq 4$, then $u$ is a root of unity.
Otherwise, $|u|\not=1$.
If $r\geq 4$ or $r=1$, then $u^4>0$.
\end{lemma}
\begin{proof}
If $u$ is a root of unity and $r>1$, then $r=(u^2+u^{-2})^2\leq |u|^4+2+|u|^{-4}=4$.
It is easy to see that $u$ is indeed a root of unity if $r\leq4$.
If $r\geq 4$ or $r=1$, then we have $u^4>0$ from (\ref{potts}).
\end{proof}

For a matrix $W \in \matX$, we define
\[
E(W)=\{ \frac{|W(x,y)|}{|W(x,x)|} \mid x,y \in X\} \subset \RR_{>0}.
\]
Then
\begin{equation}  \label{eq:6000}
E(W_1 \otimes W_2)=E(W_1)E(W_2)
\end{equation}
holds for any matrices $W_1,W_2$ with nonzero entries.

For the remainder of this section,
let $W_{H,u,a}$, $W'_{H,u,b}$ be the spin models given in
Theorem~\ref{thm:main1}(i) and (ii), respectively.
This means that $m$ is an even positive integer,
$a$ is a primitive $2m^2$-th root of unity,
$b$ is an $m^2$-th root of unity,
and $H$ is a Hadamard matrix of order $r$.


\begin{lemma} \label{lemma:5.4}
We have
\[
E(W_{H,u,a})=E(W'_{H,u,b})=
\begin{cases}
\{1, |u|^{-4}, |u|^{-3}\} &\text{if $r>4$}, \\
\{1\} &\text{otherwise}.
\end{cases}
\]
\end{lemma}
\begin{proof}
Immediate from Theorem~\ref{thm:main1} and Lemma~\ref{lemma:5.2}.
\end{proof}



\begin{lemma} \label{lemma:5.5}
\begin{itemize}
\item[{\rm (i)}] 
Suppose $r \geq 4$ or $r=1$. Then the entries of $W_{H,u,a}$, $W'_{H,u,b}$ which have absolute value $1$ are
$2m^2$-th roots of unity, $m^2$-th roots of unity, respectively.
Moreover, $W_{H,u,a}$ contains a primitive $2m^2$-th root of unity as one of its entries.
\item[{\rm (ii)}] Suppose $r=2$, and put $\nu=\LCM(2m^2,16)$, 
$\nu'=\LCM(m^2,16)$. 
Then the entries of $W_{H,u,a}$, $W'_{H,u,b}$ are 
$\nu$-th roots of unity,
$\nu'$-th roots of unity, respectively.
Moreover, $W_{H,u,a}$ contains a primitive 
$\nu$-th root of unity as one of its entries.
\end{itemize}
\end{lemma}
\begin{proof}
Firstly, suppose $r>4$.
From Lemma~\ref{lemma:5.2}, the entries of $W_{H,u,a}$, $W'_{H,u,b}$ with absolute value $1$ are
\begin{eqnarray}
\pm a^{2m(\ell-\ell')(i-j)+\epsilon(i,j)}  & &  (i-j: \text{odd}), \label{eq:44} \\
\pm \eta^{(\ell-\ell')(i-j)} b^{\delta(i,j)} & &  (i-j: \text{odd}),  \nonumber
\end{eqnarray}
which are $2m^2$-th roots of unity, $m^2$-th roots of unity, respectively.
Putting $i=1$, $j=\ell=\ell'=0$ in (\ref{eq:44}),
we obtain $a^{1+m}$ which is a primitive $2m^2$-th root of unity.


Next, suppose $r \leq 4$.
Then the entries of $W_{H,u,a}$, $W'_{H,u,b}$ are given by 
\begin{eqnarray}
v a^{2m(\ell-\ell')(i-j)+\epsilon(i,j)}  & &  (v \in \{u^3, -u^{-1}, \pm 1\}), \label{eq:0916-1} \\
v \eta^{(\ell-\ell')(i-j)} b^{\delta(i,j)} & & (v \in \{u^3, -u^{-1}, \pm 1\}), \label{eq:0916-2}
\end{eqnarray}
respectively,
all of which are roots of unity.

If $r=4$ or $1$, then from (\ref{potts}), $u^4=1$.
From (\ref{eq:0916-1}), (\ref{eq:0916-2}),
the entries of $W_{H,u,a}$, $W'_{H,u,b}$ are 
$2m^2$-th roots of unity, $m^2$-th roots of unity, respectively.
Putting $i=1$, $j=\ell=\ell'=0$ in (\ref{eq:0916-1}),
we obtain $a^{1+m}$ which is a primitive $2m^2$-th root of unity.

Finally, suppose $r=2$.
Since $u$ is a primitive $16$-root of unity by (\ref{potts}),
the expressions in (\ref{eq:0916-1}), (\ref{eq:0916-2}) are 
$\nu$-th roots of unity,
an $\nu'$-th roots of unity, respectively.
Putting $v=u^3$, $i=1$, $j=\ell=\ell'=0$ in (\ref{eq:0916-1}),
we obtain $u^3a^{1+m}$ which is a primitive $\nu$-th root of unity.
\end{proof}

For $S \in \matX$,
we denote by $\mu(S)$ the least common multiple of the orders of the entries of $S$
which have a finite order.
If none of the entries of $S$ has a finite order, 
then we define $\mu(S)=\infty$.
For a nonzero complex number $\zeta$,
we denote by the same symbol $\mu(\zeta)$ the order of $\zeta$
if $\zeta$ has a finite order.

\begin{lemma}  \label{lemma:6000}
Suppose $m\equiv 0 \pmod 4$.
Then for $W=W_{H,u,a}$ or $W=W'_{H,u,b}$, we have
$\mu(W) \mid 2m^2$.
\end{lemma}
\begin{proof}
Immediate from Lemma~\ref{lemma:5.5}.
\end{proof}

In Table~\ref{tab:1}, we summarize the properties of
$W=W_{H,u,a}$, $W'_{H,u,b}$ obtained 
from Lemmas~\ref{lemma:5.2}, \ref{lemma:5.4}, and 
\ref{lemma:6000}.

\begin{table}
\begin{tabular}{|c|c|c|l|l|c|}
\hline
$W$ & index & size & $r$ & $\mu(W)$ & $E(W)$ \\
\hline
\hline
$W_{H,u,a}$ & $m$    & $m^2r$ & $r=1$  & $2m^2$              & $\{1\}$ \\
            &        &        & $r=2$  & 
$\mu(W)|\LCM(2m^2,16)$ & $\{1\}$ \\
            &        &        & $r=4$  & $2m^2$              & $\{1\}$ \\
            &        &        & $r>4$  & $2m^2$              & $\{1,|u|^{-4},|u|^{-3}\}$ \\
\hline
$W'_{H,u,b}$ & $1$    & $m^2r$ & $r=1$ & $\mu(W)|m^2$              & $\{1\}$ \\
             &        &        & $r=2$ & 
$\mu(W)|\LCM(m^2,16)$ & $\{1\}$ \\
             &        &        & $r=4$ & $\mu(W)|m^2$              & $\{1\}$ \\
             &        &        & $r>4$ & $\mu(W)|m^2$              & $\{1,|u|^{-4},|u|^{-3}\}$ \\
\hline
\end{tabular}
\caption{Summary of Properties}
\label{tab:1}
\end{table}


For $W \in \matX$ and for a permutation $\sigma$ of $X$,
we define $W^{\sigma}$ by
$W^{\sigma}(\alpha,\beta)=W(\sigma(\alpha),\sigma(\beta))$ for $\alpha,\beta \in X$.
Observe that if $W$ is a spin model,
then $W^{\sigma}$ is also a spin model.
If $W$ is a spin model, then from (\ref{type2}), (\ref{type3}), $-W$ and $\pm\sqrt{-1}W$ are also spin models.
Two spin models $W_1$, $W_2$ are said to be {\em equivalent} if
$cW_1^{\sigma}=W_2$ for some permutation $\sigma$ of $X$ and a complex number $c$ with $c^4=1$.

Two Hadamard matrices are said to be {\em equivalent \/}
if one can be obtained from the other by negating rows 
and columns, or and permuting rows and columns.


\begin{lemma}  \label{lemma:5.01}
Let $H_1$, $H_2 \in \matY$ be equivalent Hadamard matrices.
Then $W_{H_1,u,a}$ is equivalent to $W_{H_2,u,a}$,
and $W'_{H_1,u,b}$ is equivalent to $W'_{H_2,u,b}$.
\end{lemma}
\begin{proof}
Let $(W_1, W_2, c, g)=(W_{H_1,u,a}, W_{H_2,u,a}, a, \epsilon)$ or $(W'_{H_1,u,b}, W'_{H_2,u,b}, b, \delta)$.

If $H_2$ is obtained by a permutation of columns of $H_1$,
then there exists a permutation $\pi$ of $Y$ such that $H_2(x,\pi(y))=H_1(x,y)$
for all $x,y \in Y$.
We define a permutation $\sigma$ of $X$ by
\begin{align*}
  \sigma((i,\ell,x))&=
  \begin{cases}
    (i,\ell,\pi(x)) & \text{if $i$ is odd,} \\
    (i,\ell,x) & \text{otherwise.}
  \end{cases}
\end{align*}
Then for $\alpha=(i,\ell,x), \beta=(j,\ell',y)\in X$,
\begin{align*}
W_2^{\sigma}(\alpha,\beta)
&= W_2({\sigma}(\alpha),{\sigma}(\beta))
\nexteq
\begin{cases}
  c^{g(i,j)} S_{ij}(\ell,\ell') A_u(x,y)
    &\text{if $i\equiv j\equiv0\pmod{2}$,}\\
  c^{g(i,j)} S_{ij}(\ell,\ell') H_2(x,\pi(y))
    &\text{if $i\equiv j+1\equiv0\pmod{2}$,}\\
  c^{g(i,j)} S_{ij}(\ell,\ell') H_2^T(\pi(x),y)
    &\text{if $i+1\equiv j\equiv0\pmod{2}$,}\\
  c^{g(i,j)} S_{ij}(\ell,\ell') A_u(\pi(x),\pi(y))
    &\text{if $i\equiv j\equiv1\pmod{2}$}
\end{cases}
\nexteq
\begin{cases}
  c^{g(i,j)} S_{ij}(\ell,\ell') A_u(x,y)
    &\text{if $i\equiv j\equiv0\pmod{2}$,}\\
  c^{g(i,j)} S_{ij}(\ell,\ell') H_1(x,y)
    &\text{if $i\equiv j+1\equiv0\pmod{2}$,}\\
  c^{g(i,j)} S_{ij}(\ell,\ell') H_1^T(x,y)
    &\text{if $i+1\equiv j\equiv0\pmod{2}$,}\\
  c^{g(i,j)} S_{ij}(\ell,\ell') A_u(x,y)
    &\text{if $i\equiv j\equiv1\pmod{2}$}
\end{cases}
\nexteq
W_1(\alpha,\beta).
\end{align*}

If $H_2$ is obtained by a permutation of rows of $H_1$,
then there exists a permutation $\pi'$ of $Y$ such that $H_2(\pi'(x),y)=H_1(x,y)$
for all $x,y \in Y$.
We define a permutation $\sigma'$ of $X$ by
\begin{align*}
  \sigma'((i,\ell,x))&=
  \begin{cases}
    (i,\ell,\pi'(x)) & \text{if $i$ is even,} \\
    (i,\ell,x) & \text{otherwise.}
  \end{cases}
\end{align*}
Similar calculation shows $W_2^{\sigma'}(\alpha, \beta)=
W_1(\alpha, \beta)$.

If $H_2$ is obtained by negating a column $y_1$ of $H_1$,
then $H_2(x,y_1)=-H_1(x,y_1)$, $H_2(x,y)=H_1(x,y)$ 
for all $x\in Y$ and $y\in Y-\{y_1\}$.
We define a permutation $\rho$ of $X$ by
\[
  \rho((i,\ell,x))=
  \begin{cases}
    (i,\ell+\delta_{x,y_1}\frac{m}{2},x) & \text{if $i$ is odd,} \\
    (i,\ell,x) & \text{otherwise}.
  \end{cases}
\]
Note that $S_{ij}(\ell,\ell')=(-1)^{i-j} S_{ij}(\ell+\frac{m}{2},\ell')
=(-1)^{i-j} S_{ij}(\ell,\ell'+\frac{m}{2})$.
Thus for $\alpha=(i,\ell,x), \beta=(j,\ell',y)\in X$,
\begin{align*}
& W_2^{\rho}(\alpha,\beta) 
\\ &=
W_2({\rho}(\alpha),{\rho}(\beta))
\nexteq
\begin{cases}
  c^{g(i,j)} S_{ij}(\ell,\ell') A_u(x,y) 
    & \text{if $i\equiv j\equiv0\pmod{2}$,} \\
  c^{g(i,j)} S_{ij}(\ell,\ell'+\delta_{y,y_1}\frac{m}{2}) H_2(x,y) 
    & \text{if $i\equiv j+1\equiv0\pmod{2}$,} \\
  c^{g(i,j)} S_{ij}(\ell+\delta_{x,y_1}\frac{m}{2},\ell') H_2^T(x,y) 
    & \text{if $i+1\equiv j\equiv0\pmod{2}$,} \\
  c^{g(i,j)} S_{ij}(\ell+\delta_{x,y_1}\frac{m}{2},
  \ell'+\delta_{y,y_1}\frac{m}{2}) A_u(x,y) 
    & \text{if $i\equiv j\equiv1\pmod{2}$} 
\end{cases}
\nexteq
\begin{cases}
  c^{g(i,j)} S_{ij}(\ell,\ell') A_u(x,y) & \text{if $i\equiv j\equiv0\pmod{2}$,} \\
 (-1)^{\delta_{y,y_1}}
  c^{g(i,j)} S_{ij}(\ell,\ell') H_2(x,y) 
   & \text{if $i\equiv j+1\equiv0\pmod{2}$,} \\
 (-1)^{\delta_{x,y_1}}
 c^{g(i,j)} S_{ij}(\ell,\ell') H_2^T(x,y) 
   & \text{if $i+1\equiv j\equiv0\pmod{2}$,} \\
 c^{g(i,j)} S_{ij}(\ell,\ell') A_u(x,y) 
   & \text{if $i\equiv j\equiv1\pmod{2}$} 
\end{cases}
\nexteq
\begin{cases}
  c^{g(i,j)} S_{ij}(\ell,\ell') A_u(x,y) & \text{if $i\equiv j\equiv0\pmod{2}$,} \\
  c^{g(i,j)} S_{ij}(\ell,\ell') H_1(x,y) 
    & \text{if $i\equiv j+1\equiv0\pmod{2}$,} \\
  c^{g(i,j)} S_{ij}(\ell,\ell') H_1^T(x,y) 
    & \text{if $i+1\equiv j\equiv0\pmod{2}$,} \\
  c^{g(i,j)} S_{ij}(\ell,\ell') A_u(x,y) 
    & \text{if $i\equiv j\equiv1\pmod{2}$} 
\end{cases}
\nexteq
W_1(\alpha,\beta).
\end{align*}

If $H_2$ is obtained by negating a row $x_1$ of $H_1$,
then $H_2(x_1,y)=-H_1(x_1,y)$, $H_2(x,y)=H_1(x,y)$ 
for all $x\in Y-\{x_1\}$ and $y \in Y$.
We define a permutation $\rho'$ of $X$ by
\[
  \rho'((i,\ell,x))=
  \begin{cases}
    (i,\ell+\delta_{x,x_1}\frac{m}{2},x) & \text{if $i$ is even,} \\
    (i,\ell,x) & \text{otherwise}.
  \end{cases}
\]
Similar calculation shows $W_2^{\rho'}(\alpha, \beta)=
W_1(\alpha, \beta)$.
\end{proof}


\section{Decomposability}\label{sec:5}


\begin{lemma} \label{lemma:5.3}
Let $S_1, S_2$ be finite subsets of positive real numbers.
Suppose $1\in S_1 \cap S_2$ and $|S_1S_2|=3$. Then
\[
(|S_1|,|S_2|) \in \{(2,2),(1,3),(3,1)\}.
\]
If $|S_1|=|S_2|=2$, then $S_1S_2=\{1,a,a^2\}$ or $\{1,a,a^{-1}\}$ 
for some positive real number $a\neq1$.
\end{lemma}
\begin{proof}
By way of contradiction, 
we prove that if $|S_1|\geq 3$ and $|S_2|\geq 2$
then $|S_1S_2| > 3$.
Since $S_1 \cup S_2 \subset S_1S_2$, we obtain $S_2 \subset S_1=S_1S_2$.
Let 
$S_1=\{1,\lambda,\mu\} \ (\lambda,\mu\not=1, \lambda\not=\mu)$.
Then we may put $S_2=\{1,\lambda\}$ without loss of generality.
Then we have $\lambda^2 \in S_1S_2=S_1$, so $\mu=\lambda^2$
%
and $S_1S_2=\{1,\lambda,\lambda^2,\lambda^3 \}$.
This implies $|S_1S_2|=4$, a contradiction.

Suppose $|S_1|=|S_2|=2$. Then $S_1=\{1,a\}$, $S_2=\{1,b\}$ 
for some $a,b\neq1$.
Then $|S_1S_2|=3$ implies $a=b$ or $a=b^{-1}$.
\end{proof}


\begin{lemma} \label{lemma:5.10}
Let $A \in \textsl{Mat}_{Z_1}(\CC^*)$ be a matrix all of whose entries are roots of unity. 
Let $B \in \textsl{Mat}_{Z_2}(\CC^*)$ be a matrix which satisfies
$\mu(B)<\infty$.
Then $\mu(A \otimes B)$ is a divisor of 
$\LCM(\mu(A),\mu(B))$.
\end{lemma}
\begin{proof}
Let $Z_2'=\{(x_2,y_2)\in Z_2\times Z_2
\mid o(B(x_2,y_2)) < \infty \}$. Then 
\[
\mu(A \otimes B) = 
\LCM ( \{ o(A(x_1,y_1)B(x_2,y_2)) \mid x_1, y_1 \in Z_1, (x_2, y_2) \in Z_2' \}),
\]
which is a divisor of $\LCM(\mu(A),\mu(B))$.
\end{proof}


Some examples of spin models are listed in Section $1$, i.e.,
Potts model, non-symmetric Hadamard models, and Hadamard models.
We remark that non-symmetric Hadamard models and Hadamard models
are special cases of spin models given in Theorem~\ref{thm:main1} (i), (ii), respectively.
In addition to these examples, the following spin models are known.

\paragraph{Spin models on finite abelian groups.}
Bannai-Bannai-Jaeger \cite{BBJ} gives solutions to modular invariance equation
for finite abelian groups, and every solution gives a spin model.
Let $U$ be a finite abelian group, and
$e=\exp(U)$ denote the exponent of $U$.
Let $\{\chi_a\mid a\in U\}$ be the set of characters of $U$
with indices chosen so that $\chi_a(b)=\chi_b(a)$ for all
$a, b \in U$.
Let $U=U_1 \oplus \dots \oplus U_h$ be a
decomposition of $U$ into a direct sum of cyclic groups $U_1, U_2, \dots, U_h$.
For each $i\in\{1,2,\dots,h\}$ let $a_i$ be a generator and $n_i$ be the order of the cyclic group $U_i$.
For each $x\in U$, we define the matrix $A_x \in \textsl{Mat}_U(\CC)$ by
\[
A_x(\alpha,\beta)=\delta_{x,\beta-\alpha} \quad \text{($\alpha,\beta \in U$)}.
\]
For any $x=\sum_{i=1}^h x_ia_i$ $(0\leq x_i < n_i)$, let
\begin{equation}  \label{eq:450}
t_x=t_0 \prod_{i=1}^h \eta_i^{x_i} \chi_{a_i}(a_i)^{\frac{x_i(x_i-1)}{2}} 
\prod_{1\leq \ell<k\leq h} \chi_{a_{\ell}}(a_k)^{x_{\ell} x_k},
\end{equation}
where $\eta_i^{n_i}=\chi_{a_i}(a_i)^{-\frac{n_i(n_i-1)}{2}}$ and
\begin{equation}  \label{eq:451}
t_0^2=D^{-1} \sum_{x \in U} \prod_{j=1}^h \eta_j^{-x_j} \chi_{a_j}(a_j)^{-\frac{x_j(x_j-1)}{2}}
\prod_{1\leq \ell<k\leq h} \chi_{a_{\ell}}(a_k)^{-x_{\ell} x_k},
\end{equation}
where $D^2=|U|$.
Let $\theta_x=t_x/t_0$ for any $x \in U$.
Then, for any $x\in U$, $\theta_x$ is a root of unity and $\theta_x^{2e}=1$.
Especially, we get
\begin{equation}  \label{eq:452}
\theta_x^{2|U|}=1.
\end{equation}
The matrix
\begin{equation}  \label{eq:W}
W=\sum_{x \in U} t_xA_x.
\end{equation}
is a spin model.

\paragraph{Jaeger's Higman-Sims model.}
In \cite{J}, F.~Jaeger constructed a spin model $W_J$ on the Higman-Sims graph
of size $100$. 
We denote by $A$ the adjacency matrix of the Higman-Sims graph.
We put $W_J=-\tau^5 I-\tau A+\tau^{-1} (J-A-I)$, where $\tau$ satisfies $\tau^2+\tau^{-2}=3$.
Then $W_J$ is a symmetric spin model.

Now every known spin model belongs to one of the following five families:
\begin{itemize}
\item[(a)] $A_u$: Potts model of size $r \geq 2$.
If $r=2$, then $\mu(A_u)=16$. If $r=4$, then $\mu(A_u)=2$ or $4$.
If $r=2,4$, then $E(A_u)=\{1\}$.
If $r>4$, then $E(A_u)=\{1, |u|^{-4}\}$, and hence $|E(A_u)|=2$.
\item[(b)] $W_U$: spin model on a finite abelian group $U$.
We have various kinds of indices and $E(W_U)=\{1\}$.
\item[(c)] $W_J$: Jaeger's Higman-Sims model of size $100$.
We have $E(W_J)=\{1,\tau^{-4},\tau^{-6}\}$ with $\tau^2+\tau^{-2}=3$.
and hence $|E(W_J)|=3$.
\item[(d)] $W_{H,u,a}$: spin models given in Theorem~\ref{thm:main1}(i).
\item[(e)] $W'_{H,u,b}$: spin models given in Theorem~\ref{thm:main1}(ii).
\end{itemize}

By way of contradiction, we now give a proof of Theorem~\ref{thm:main2}.
Let $H$ be a Hadamard matrix of order $r>4$.
Let $s$ be a positive integer and $a$ a primitive $2^{2s+1}$-th root of unity.
For the remainder of this section, we denote by $W$ the spin model $W_{H,u,a}$ 
given in Theorem~\ref{thm:main1} (i) of index $2^s$.
By Lemma~\ref{lemma:5.4} we obtain 
\begin{equation}  \label{1004-1}
E(W)=\{1,|u|^{-4}, |u|^{-3}\}.
\end{equation}

We assume that
\begin{equation}  \label{eq:46}
W= W_1 \otimes W_2 \otimes \dots \otimes W_v,
\end{equation}
where each of $W_1, W_2, \dots, W_v$ is a known spin model listed in (a)--(e)
and their sizes are not equal to $1$.
Since $|E(W)|=3$ from (\ref{1004-1}),
using Lemma~\ref{lemma:5.3} we may assume without loss of generality
\[
(|E(W_1)|,|E(W_2)|,\dots,|E(W_v)|)=
(1,\dots,1,2,2) \ \text{or} \ 
(1,\dots,1,3).
\]

A known spin model $W'$ with $|E(W')|=1$ belongs to the family
(b) or to the families (a), (d) and (e) with $r\leq4$.
Therefore, (\ref{eq:46}) can be reduced to the following cases:
\begin{eqnarray}
W &=& W_1 \otimes W_2 \otimes W_3 \ \text{with} \ E(W_1)=\{1\}, |E(W_2)|=|E(W_3)|=2, \label{case2} \\
W &=& W_1 \otimes W_2 \ \text{with} \ E(W_1)=\{1\}, |E(W_2)|=3, \label{case1}
\end{eqnarray}
where in (\ref{case2}), (\ref{case1}),
$W_1$ is a tensor product of spin models on finite abelian groups
and spin models in the families (a), (d) and (e) with $r\leq 4$.
Note that $W_1$ could possibly be of size 1 in (\ref{case2}).

First, we treat the case (\ref{case2}).
Then Lemma~\ref{lemma:5.3} implies
$E(W_2 \otimes W_3)=\{1,\beta,\beta^2\}$, or $\{1,\beta,\beta^{-1}\}$
for some $\beta$.
On the other hand, $E(W_2 \otimes W_3)=E(W_1)E(W_2 \otimes W_3)
=E(W)=\{1,|u|^{-4},|u|^{-3}\}$ by (\ref{1004-1}). This is a contradiction.

Next, we treat the case (\ref{case1}).
We have $E(W_2)=E(W_1)E(W_2)=E(W)=\{1,|u|^{-4},|u|^{-3}\}$ from (\ref{1004-1}).
Since $\{1,|u|^{-4},|u|^{-3}\}\neq\{1, \tau^{-4}, \tau^{-6}\}$,
$W_2$ cannot be the spin model (c).
Therefore, $W_2$ belongs to the family (d) or (e).
This means $W_2=W_{H',u',a'}$ or $W_2=W'_{H',u',b'}$,
where $H'$ is a Hadamard matrix of order $r'=(u'^2+u'^{-2})^2$.
Since $|E(W_2)|=3$, Lemma~\ref{lemma:5.4} implies $r'>4$
and $E(W_2)=\{1,|u'|^{-4},|u'|^{-3}\}$.
Then we have $|u'|=|u|$, as $E(W)=E(W_2)$.
Now the second part of Lemma~\ref{lemma:5.2} implies $u^4>0$ and
$u'^4>0$, hence
\begin{equation}  \label{0921}
u^4=u'^4,
\end{equation}
and further $r=r'$ by (\ref{potts}).
Therefore the size of $W_2$ is $2^{2s'}r$ for some integer $s'$ with $0<s'<s$, and the size of $W_1$ is $2^{2(s-s')}$.
In particular, we obtain $s>1$.



Since the tensor product of spin models on finite abelian groups 
is also a spin model on a finite abelian group,
we may suppose that
\begin{equation} \label{eq:6002}
W_1=W_{11} \otimes W_{12} \otimes W_{13},
\end{equation}
where $W_{11}$ is a spin model on a finite abelian group $U$,
$W_{12}$ is a tensor product of spin models in the family (a) with $r\leq4$,
and $W_{13}$ is a tensor product of spin models 
in the families (d) and (e) with $r\leq4$.

We put $|U|=2^{n_1}$. Since the size $2^{n_1}$ of $W_{11}$
cannot exceed that of $W_1$, we have $n_1 \leq 2(s-s')$.
Then the size of $W_{12} \otimes W_{13}$ is $2^{2(s-s')-n_1}$.
The diagonal entry of $W_{11}$ is a complex number $t_0$ given by (\ref{eq:451}).
The diagonal entries of $W_{12}$, $W_{13}$ are $16$-th roots of unity.
We denote by $\kappa_2$, $\kappa_3$ the diagonal entries of $W_{12}$, $W_{13}$, respectively.
Comparing the diagonal entries of (\ref{eq:6002}),
we have $u^3=t_0 \kappa_2\kappa_3 u'^3$, thus
\begin{equation} \label{eq:2010}
W= (t_0^{-1}W_{11}) \otimes (\kappa_2^{-1}W_{12}) \otimes (\kappa_3^{-1}W_{13}) \otimes (u^3 u'^{-3}W_2).
\end{equation}
From (\ref{eq:452}), we have 
\begin{equation} \label{eq:5001}
\mu(t_0^{-1}W_{11}) \mid 2^{n_1+1}.
\end{equation}
From (a), we have
\begin{equation} \label{eq:1008}
\mu(\kappa_2^{-1}W_{12}) \mid 2^4.
\end{equation}
From (a) and Lemma~\ref{lemma:6000}, we have 
\begin{equation} \label{eq:5002}
\mu(\kappa_3^{-1}W_{13}) \mid 2^{2(s-s')-n_1+1}.
\end{equation}
Since $W_2$ is a spin model belonging to the family (d) or (e),
Lemma~\ref{lemma:5.5} and (\ref{0921}) imply
\begin{equation}  \label{eq:1010}
\mu(u^3u'^{-3}W_2) \mid 2^{2s'+1}.
\end{equation}
From (\ref{eq:2010})--(\ref{eq:1010}) and Lemma~\ref{lemma:5.10},
we have 
\[
\mu(W) \mid \LCM (2^{n_1+1}, 2^4, 2^{2(s-s')-n_1+1}, 2^{2s'+1}).
\]
Since $n_1< 2s$, we have $\max(n_1+1, 4, 2(s-s')-n_1+1, 2s'+1) \leq 2s$.
This implies $\mu(W) \mid 2^{2s}$,
which contradicts Lemma~\ref{lemma:5.5} (i).



\section{Spin models in Theorem~\ref{thm:main1} with $r\leq 4$}

In this section, we treat the case of $r\leq 4$ in Theorem~\ref{thm:main2}.
We show that if $r=1, 4$ in Theorem~\ref{thm:main2}, then $W_{H,u,a}$ is not new.

If $r=4$ in Theorem~\ref{thm:main1} (i), then
$W_{H,u,a}$ is a tensor product of
a Hadamard matrix of order $4$ and $W_{(1),u,a}$.
Indeed, up to equivalence, there is a unique Hadamard matrix of order $r=4$.
By Lemma~\ref{lemma:5.01}, we may assume without loss of generality
\[
H=\begin{pmatrix}
1 & -1 & -1 & -1 \\
-1 & 1 & -1 & -1 \\
-1 & -1 & 1 & -1 \\
-1 & -1 & -1 & 1
\end{pmatrix}.
\]
Then $A_u=u^3H$ with $(u^2+u^{-2})^2=4$.
Therefore we have $W_{H,u,a}=H \otimes W_{(1),u,a}$.
Similarly, a spin model $W'_{H,u,b}$ in Theorem~\ref{thm:main1} (ii)
can be decomposed as $H \otimes W'_{(1),u,b}$.


\begin{lemma} \label{lemma:0901}
Let $m\equiv 0 \pmod{4}$.
Let $W_{(1),u,a}$ be a spin model given in Theorem~{\rm\ref{thm:main1}} of index $m$,
where $u^4=1$ and $a$ is a primitive $2m^2$-th root of unity.
Then $W_{(1),u,a}$ is equivalent to $W_{(1),1,au^3}$.
\end{lemma}
\begin{proof}
First we assume that $u=-1$.
Then $a^{\epsilon(i,j)}(-1)^{i-j-1}=-(-a)^{\epsilon(i,j)}$ holds
for all $i,j \in \ZZ_{m^2}$.
From this, we have $W_{(1),-1,a}=-W_{(1),1,-a}$.
Therefore $W_{(1),-1,a}$ is equivalent to $W_{(1),1,-a}$.

Next we assume that $u^2=-1$.
Since $m\equiv 0 \pmod{4}$, we have
\begin{eqnarray*}
u (au^3)^{\epsilon(i,j)} &=
  \begin{cases}
  a^{\epsilon(i,j)} u & \text{if $i-j$ is even,} \\
  a^{\epsilon(i,j)} & \text{if $i-j$ is odd}.
  \end{cases}
\end{eqnarray*}
From this, we have $uW_{(1), 1, au^3} = W_{(1), u, a}$.
Therefore $W_{(1),u,a}$ is equivalent to $W_{(1),1,au^3}$.
\end{proof}


\begin{lemma} \label{lemma:5.7}
Let $m$ be even, and $\xi$ be a primitive $2m^2$-th root of unity.
Then we have
\begin{equation}  \label{eq:0901}
\sum_{x=0}^{m^2-1} \xi^{-x(x-m)}=m.
\end{equation}
\end{lemma}
\begin{proof}
If (\ref{eq:0901}) holds for $\xi=\exp(2\pi\sqrt{-1}/(2m^2))$,
then by considering the action of the Galois group,
we see that (\ref{eq:0901}) holds for any primitive $2m^2$-th root of unity $\xi$.
Therefore we may assune $\xi=\exp(2\pi\sqrt{-1}/(2m^2))$ without loss of generality.
Since $m$ is even, we may write $m=2k$. Then
\begin{align*}
\sum_{x=0}^{m^2-1} \xi^{-x(x-m)}
&= \sum_{x=0}^{m^2-1} \xi^{-( (x-k)^2-k^2 )} \\
&= \xi^{k^2} \sum_{x=0}^{m^2-1} \xi^{-(x-k)^2} \\
&= \frac{\xi^{k^2}}{2} \sum_{x=0}^{m^2-1} (\xi^{-(x-k)^2}+\xi^{-(x-k+m^2)^2}) \\
&= \frac{\exp(\pi\sqrt{-1}/4)}{2} \sum_{x=0}^{2m^2-1} \xi^{-(x-k)^2} \\
&= \frac{1+\sqrt{-1}}{2\sqrt{2}} \sum_{x=0}^{2m^2-1} \xi^{-x^2}.
\end{align*}
Now the result follows from \cite[Theorem 99]{Na}.
\end{proof}

Of particular interest among spin models on finite abelian groups
are spin models on finite cyclic groups.
The spin model defined below is a special case of spin models on 
finite cyclic groups constructed by \cite{BB0}.
Let $m$ be even, and $a$ be a primitive $2{m^2}$-th root of unity.
We restrict (\ref{eq:450}) and (\ref{eq:451}) to $\ZZ_{m^2}$, 
that is, $h=1$. 
In (\ref{eq:450}) and (\ref{eq:451}), 
we put $\eta_1=a^{-m+1}$, $\chi_{a_1}(a_1)=a^2$.
Then (\ref{eq:450}) and (\ref{eq:451}) become
\begin{eqnarray}
t_x &=& t_0 a^{x(x-m)} \quad (x \in \ZZ_{m^2}), \label{eq:453} \\
t_0^2 &=& m^{-1} \sum_{x=0}^{m^2-1} a^{-x(x-m)}=1, \label{eq:454}
\end{eqnarray}
respectively, where we used
Lemma~\ref{lemma:5.7} in (\ref{eq:454}).
Thus we may take $t_0=1$.
Then the matrix $W$ given in (\ref{eq:W}) has entries
\begin{equation} \label{eq:1008-hakata}
W(\alpha, \beta)=a^{ (\beta-\alpha)(\beta-\alpha-m) } 
\quad \text{($\alpha,\beta \in \ZZ_{m^2}$)}.
\end{equation}
We note that this spin model $W$ on $\ZZ_{m^2}$
was constructed originally in \cite[Theorem 2]{BB}.

%


\begin{proposition} \label{prop:6.4}
Let $m\equiv 0 \pmod{4}$.
Let $W_{(1),u,a}$ be a spin model given in Theorem~{\rm\ref{thm:main1} (i)} of index $m$,
where $u^4=1$ and $a$ is a primitive $2m^2$-th root of unity.
Then $W_{(1),u,a}$ is equivalent to $W$ defined in {\rm(\ref{eq:1008-hakata})}.
\end{proposition}
\begin{proof}
From Lemma~\ref{lemma:0901} it is sufficient to prove that 
$W_{(1),1,au^3}$ is equivalent to $W$.
By assumption, $m=4k$ for some positive integer $k$.
Since $a^{8k^2}$ is a primitive $4$-th root of unity, 
there exists $t\in\ZZ_4$ such that $u^3=a^{8k^2t}$.
We define a bijection $\psi: \ZZ_m^2 \to \ZZ_{m^2}$ by
\[
\psi(i,\ell)=(4k^2t+1)i+4k\ell
\]
for $(i,\ell)\in \ZZ_m^2$.
Then 
for all $i,j, \ell,\ell' \in \ZZ_m$,
\begin{align*}
&(\psi(j,\ell')-\psi(i,\ell))(\psi(j,\ell')-\psi(i,\ell)-m)
\\ &=
((4k^2t+1)(j-i)+4k(\ell'-\ell))((4k^2t+1)(j-i)+4k(\ell'-\ell)-4k)
\\ &=
(8k^2t+1)(8k(\ell-\ell')(i-j)+(i-j)^2+4k(i-j))
\\ &\quad+
32k^2\left(-kt(j-i)(l'-l)
+\frac{kt(j-i)(kt(j-i)+1)}{2}\right.
\\ &\qquad\qquad\quad\left.
+\frac{(l'-l)(l'-l-1)}{2}\right)
\\ &\equiv (8k^2t+1)(8k(\ell-\ell')(i-j)+(i-j)^2+4k(i-j))\pmod{32k^2}.
\end{align*}
Thus 
\begin{align*}
W( \psi(i,\ell),\psi(j,\ell') )&=
a^{(\psi(j,\ell')-\psi(i,\ell))(\psi(j,\ell')-\psi(i,\ell)-m)}
\nexteq
a^{(8k^2t+1)(8k(\ell-\ell')(i-j)+(i-j)^2+4k(i-j))}
\nexteq
(au^3)^{2m(\ell-\ell')(i-j)+(i-j)^2+m(i-j)}
\nexteq
W_{(1),1,au^3}((i,\ell,1),(j,\ell',1)),
\end{align*}
and we conclude that $W$ is equivalent to $W_{(1),1,au^3}$.
\end{proof}

To conclude the paper, we note that the decomposability
and identification with known spin models are yet to
be determined for the following cases.
\begin{enumerate}
\item $W_{H,u,a}$: $r=1$, $m\equiv2\pmod4$,
\item $W'_{H,u,b}$: $r=1$,
\item $W_{H,u,a}$ and $W'_{H,u,b}$: $r=2$,
\item $W_{H,u,a}$ and $W'_{H,u,b}$: $r>4$ and $m$ is not
a power of $2$.
\end{enumerate}






\ifx\undefined\allcaps\def\allcaps#1{#1}\fi\newcommand{\nop}[1]{}
\providecommand{\bysame}{\leavevmode\hbox to3em{\hrulefill}\thinspace}

\end{document}